\documentclass[11pt,letterpaper]{amsart}

\usepackage{epsfig,amsmath,amsthm,amssymb,graphicx,hyperref}
\usepackage[top=1.25in, bottom=1.25in, left=1.25in, right=1.25in]{geometry}
\usepackage{caption}

\newtheorem{proposition}{Proposition}[section]
\newtheorem{theorem}[proposition]{Theorem}

\newtheorem*{theorem*}{Theorem}

\newtheorem*{thm_3.3}{Theorem 3.3}
\newtheorem*{thm_3.1}{Theorem 3.1}

\newtheorem{corollary}[proposition]{Corollary}

\theoremstyle{definition}

\newtheorem{remark}[proposition]{Remark}

\begin{document}
\title[Inequality on $t_\nu(K)$ defined by Livingston and Naik and its applications]{Inequality on $t_\nu(K)$ defined by Livingston and Naik and its applications}

\author{JungHwan Park}
\address{Department of Mathematics, Rice University MS-136\\
6100 Main St. P.O. Box 1892\\
Houston, TX 77251-1892}
\email{jp35@rice.edu}

\date{\today}


\begin{abstract}
Let $D_+(K,t)$ denote the positive $t$-twisted double of $K$. For a fixed integer-valued additive concordance invariant $\nu$ that bounds the smooth four genus of a knot and determines the smooth four genus of positive torus knots, Livingston and Naik defined $t_\nu(K)$ to be the greatest integer $t$ such that $\nu(D_+(K,t)) = 1$. Let $K_1$ and $K_2$ be any knots then we prove the following inequality : $t_\nu(K_1) + t_\nu(K_2) \leq t_\nu(K_1 \# K_2) \leq min(t_\nu(K_1) - t_\nu(-K_2), t_\nu(K_2) - t_\nu(-K_1)).$ As an application we show that $t_\tau(K) \neq t_s(K)$ for infinitely many knots and that their difference can be arbitrarily large, where $t_\tau(K)$ (respectively $t_s(K)$) is $t_\nu(K)$ when $\nu$ is Ozv\'{a}th-Szab\'{o} invariant $\tau$ (respectively when $\nu$ is normalized Rasmussen $s$ invariant).
\end{abstract}

\maketitle

%=============================================================
\section{Introduction}\label{Introduction}
Let $\nu$ be any integer-valued concordance invariant with the following properties: 
\begin{enumerate}
\item additive under connected sum,
\item $|\nu(K)| \leq g_4(K)$,
\item $\nu(T_{p,q}) = (p-1)(q-1)/2$ for $p,q>0$.
\end{enumerate}
Notice that the Ozv\'{a}th-Szab\'{o} invariant $\tau$ satisfies the above properties \cite{OzSz03}, as does the Rasmussen $s$ invariant when suitably normalized (i.e. when $\nu = -s/2$) \cite{Ra10}. Let $D_\pm(K,t)$ denote the positive or negative $t$-twisted double of $K$. Then for a fixed concordance invariant $\nu$, Livingston and Naik \cite{LN06} show that $\nu(D_+(K,t))$ is always $1$ or $0$ (see Theorem $2.1$) and define $t_\nu(K)$ to be the greatest integer $t$ such that $\nu(D_+(K,t)) = 1$. Specializing to $\tau$ and $s$, we have the following two concordance invariants $t_\tau(K)$ (respectively $t_s(K)$) which is the greatest integer $t$ where $\tau(D_+(K,t)) = 1$ (respectively $-s(D_+(K,t))/2 = 1$). Hedden and Ording \cite{HO08} show that there exist $K$ for which $t_\tau(K) \neq t_s(K)$, in particular they show that $t_\tau(T_{2,2n+1}) = 2n-1$ whereas $t_s(T_{2,3}) \geq 2$, $t_s(T_{2,5}) \geq 5$, and $t_s(T_{2,7}) \geq 8$ (In fact, it is easy to verity that $t_s(T_{2,3}) = 2$ and $t_s(T_{2,5}) = 5$ using Bar-Natan's program \cite{Bar}). This was the first example known of a knot $K$ for which $\tau(K) \neq -s(K)/2$. (Note that it is proven that $\tau \neq -s/2$ even for topologically slice knots by Livingston, \cite{Li08}.) Further they make a remark that it would be reasonable to guess that $t_s(T_{2,2n+1}) = 3n-1$, which would imply that $t_\tau(K) \neq t_s(K)$ for infinitely many different knots. Also, Hedden \cite{He07} showed that $t_\tau(K)$ does not give more information than $\tau(K)$:

\begin{theorem}\cite[Theorem $1.5$]{He07} $t_\tau(K) = 2\tau(K) -1$.
\end{theorem} 

However, $t_s(K)$ is not well understood. In this paper we show the following inequality :

\begin{theorem} Let $K_1$ and $K_2$ be any knots and $\nu$ be any integer-valued concordance invariant with properties $(1)$, $(2)$, and $(3)$ as above then the following inequality hold:
$$t_\nu(K_1) + t_\nu(K_2) \leq t_\nu(K_1 \# K_2) \leq min(t_\nu(K_1) - t_\nu(-K_2), t_\nu(K_2) - t_\nu(-K_1)).$$
\end{theorem}

We have the following as the immediate corollary:
\begin{corollary} For any positive integer $n$, there exists a knot $K_n$ such that 
$$|t_\tau(K_n)-t_s(K_n)| > n.$$
\end{corollary}
\begin{proof} Let $K_n$ be $n$ connected sum of $T_{2,5}$. Then by Theorem $1.2$ and the fact that $t_\tau(K_n) = n\cdot\tau(T_{2,5}) -1 = 4n-1$ and $t_s(T_{2,5}) \geq 5$ by Theorem $1.1$ and \cite{HO08}, the result follows.
\end{proof}

We end this section with the following remark:

\begin{remark} If we assume that $t_s(K)$ is a polynomial of $-s(K)/2$ with integer coefficients, it is easy to verity that $t_s(K) = 3\cdot (-s(K)/2) - 1$ using Theorem $1.2$. Then in the light of \cite{HO08} it would be reasonable to guess that $t_s(K) = 3\cdot (-s(K)/2) - 1$.
\end{remark}

\subsection{Acknowledgements} The author would like to thank his advisor Shelly Harvey, and also David Krcatovich for their helpful discussions.
%=============================================================
\section{Proof of the Theorem $1.2$}\label{Proof of the Theorem $1.2$}
We will denote by $TB(K)$ the maximum value of the Thurston-Bennequin number, taken over all possible Legendrian representatives of $K$. Recall the following theorem from \cite{LN06}: 
\begin{theorem}\cite[Theorem $2$]{LN06} For each knot $K$ there is an integer $t_\nu(K)$ such that:
$$\nu(D_+(K,t))=\left\{
\begin{array}{c l}      
    0 & \text{for } t > t_\nu(K)\\
    1 & \text{for } t \leq t_\nu(K),
\end{array}\right.$$ where $t_\nu(K)$ satisfies $TB(K) \leq t_\nu(K) < -TB(-K)$.
\end{theorem}

A similar result holds for $D_-(K,t)$ using $t_\nu(-K)$:
$$\nu(D_-(K,t))=\left\{
\begin{array}{c l}      
    -1 & \text{for } t \geq -t_\nu(-K)\\
    0 & \text{for } t < -t_\nu(-K),
\end{array}\right.$$ where $t_\nu(-K)$ satisfies $TB(-K) \leq t_\nu(-K) < -TB(K)$.

Now, we are ready to prove the Theorem $1.2$. The proof completely relies on Theorem $2.1$.
\begin{proof}[Proof of Theorem $1.2$]
Let $t_1$ and $t_2$ be integers and consider $D_+(K_1,t_1) \# D_+(K_2,t_2)$ and $D_+(K_1 \# K_2, t_1+t_2)$. Then there is a genus one cobordism from $D_+(K_1,t_1) \# D_+(K_2,t_2)$ to $D_+(K_1 \# K_2, t_1+t_2)$ (see Figure $1$). Hence if $\nu(D_+(K_1,t_1)) = \nu(D_+(K_2,t_2)) = 1$, then $\nu(D_+(K_1 \# K_2, t_1+t_2)) = 1$. Letting $t_1 = t_\nu(K_1)$ and $t_2 = t_\nu(K_2)$, we have $\nu(D_+(K_1,t_1)) = \nu(D_+(K_2,t_2)) = 1$ by Theorem $2.1$. Using Theorem $2.1$ again we have $t_\nu(K_1) + t_\nu(K_2) \leq t_\nu(K_1 \# K_2)$. 

\begin{figure}[h]
\centering
\includegraphics[width=5.5in]{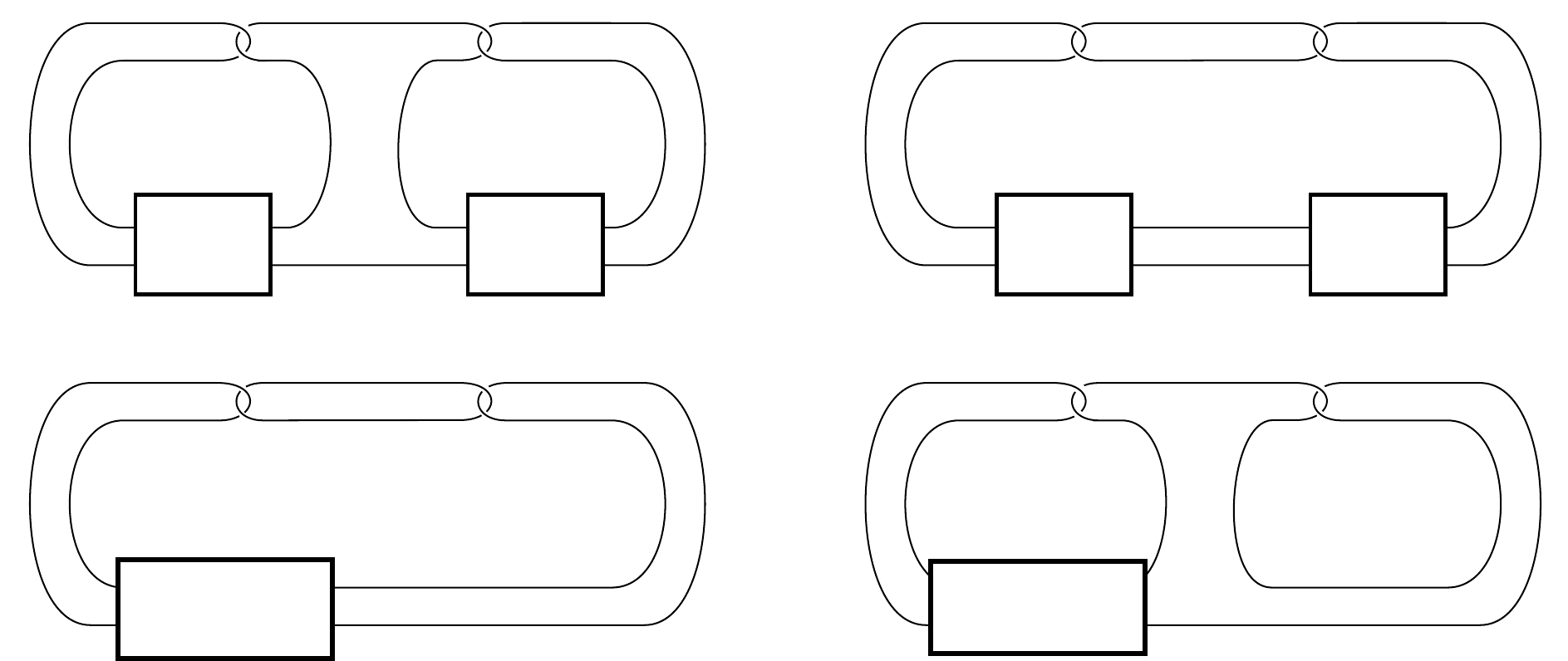}
\put(-2.85,1.7){ $\rightarrow$}
\put(-2.85,0.45){ $\rightarrow$}
\put(-5.8,0.45){ $\cong$}
\put(-5.02,1.4){ $K_1,t_1$}
\put(-3.85,1.4){ $K_2,t_2$}
\put(-2.02,1.4){ $K_1,t_1$}
\put(-0.9,1.4){ $K_2,t_2$}
\put(-5.05,0.2){ \footnotesize$K_1 \# K_2,$}
\put(-4.85,0.08){ \footnotesize$t_1+t_2$}
\put(-2.2,0.2){ \footnotesize$K_1 \# K_2,$}
\put(-2,0.08){ \footnotesize$t_1+t_2$}
\caption{A genus one cobordism from $D_+(K_1,t_1) \# D_+(K_2,t_2)$ to $D_+(K_1 \# K_2, t_1+t_2)$. The top left figure is $D_+(K_1,t_1) \# D_+(K_2,t_2)$, the top right figure is obtained from the top left figure after one band sum, the bottom left figure is obtained from the top right figure after an isotopy, and the bottom right figure is obtained from the bottom left figure after one band sum and it is isotopic to $D_+(K_1 \# K_2, t_1+t_2)$. }
\end{figure}
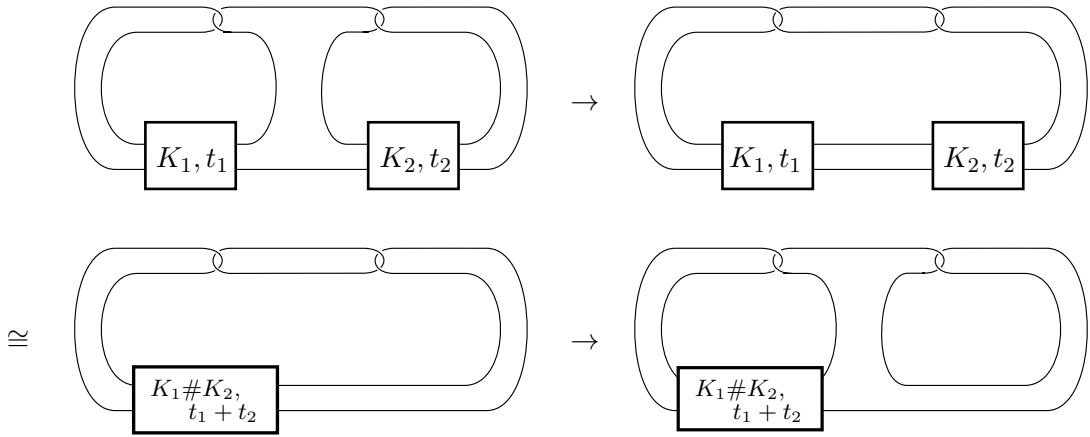

Using a similar argument, notice that there is a genus one cobordism from $D_+(K_1,t_1) \#$ $D_-(K_2,t_2)$ to $D_+(K_1 \# K_2, t_1+t_2)$ by simply changing the sign of the clasp in Figure $1$. Therefore if $\nu(D_+(K_1,t_1)) = 0$ and $\nu(D_-(K_2,t_2)) = -1$, then $\nu(D_+(K_1 \# K_2, t_1+t_2)) = 0$ by Theorem $2.1$. Letting $t_1 = t_\nu(K_1)+1$ and $t_2 = -t_\nu(-K_2)$, we have $\nu(D_+(K_1,t_1)) =0$ and $\nu(D_-(K_2,t_2)) = -1$ by Theorem $2.1$. Using Theorem $2.1$ again we have $ t_\nu(K_1)+1-t_\nu(-K_2) \geq t_\nu(K_1 \# K_2) +1$, hence $t_\nu(K_1 \# K_2) \leq t_\nu(K_1)-t_\nu(-K_2)$. Finally, by switching roles of $K_1$ and $K_2$ we also get $t_\nu(K_1 \# K_2) \leq t_\nu(K_2)-t_\nu(-K_1)$ which completes the proof.

\end{proof}
%=============================================================

\bibliographystyle{alpha}
\bibliography{knotbib}

\end{document}